\documentclass[a4paper,11pt]{article}

\usepackage[top=3.0cm,bottom=3.0cm,left=2.5cm,right=2.5cm]{geometry}

\usepackage{setspace}
\usepackage{lmodern}

\usepackage{graphicx}
\usepackage{subcaption}
\usepackage{amsfonts}
\usepackage{amssymb}
\usepackage{amsmath}
\usepackage{amsthm}
\usepackage[english]{babel}

\usepackage{hyperref}
\usepackage{cases}
\usepackage{authblk}
\usepackage{mathtools}
\usepackage{color}
\usepackage{ifpdf}
\usepackage[graphicx]{realboxes}
\usepackage{array}
\usepackage{diagbox}
\usepackage{multirow}
\usepackage{booktabs}
\usepackage{tikz}
\usetikzlibrary{calc, arrows.meta}

\newtheorem{lemma}{Lemma}[section]
\newtheorem{theorem}[lemma]{Theorem}

\newtheorem{claim}{\noindent Claim}[section]

\newtheorem{conjecture}[lemma]{Conjecture}

\begin{document}

\setstretch{1.35} 
	
\title{Colouring ($P_2\cup P_4$, diamond)-free graphs with $\omega$ colours}

\author{Hongyang Wang\thanks{Email address: hwanghj@connect.ust.hk}}

\affil{Department of Mathematics,

Hong Kong University of Science and Technology,

Clear Water Bay, Hong Kong}

\date{\today}
\maketitle

\begin{abstract}

In this paper, we determine the optimal $\chi$-binding function for $(P_2\cup P_4,\text{ diamond})$-free graphs. 
Specifically, for any graph $G$ in this class, $\chi(G)\le 4$ when $\omega(G)=2$, 
$\chi(G)\le 6$ when $\omega(G)=3$, and $\chi(G)=\omega(G)$ when $\omega(G)\ge 4$, 
where $\chi(G)$ and $\omega(G)$ denote the chromatic number and clique number of $G$, respectively. Moreover, all these bounds are sharp. 
This result extends the known chromatic bounds for $(P_2\cup P_3,\text{ diamond})$-free graphs \cite{P2P3,P6D2,P2P3DIA} by showing that $(P_2\cup P_4,\text{ diamond})$-free graphs admit the same optimal $\chi$-binding function. 
It also improves the previously known chromatic bounds obtained by Angeliya, Karthick and Huang~\cite{P2P4DIA} and by Chen and Zhang~\cite{RANCHEN} for $(P_2\cup P_4,\text{ diamond})$-free graphs.

\end{abstract}

				

\section{Introduction}

All graphs in this paper are finite and simple. We follow \cite{WEST} for undefined notations and terminology. Let $G$ be a graph with vertex set $V(G)$ and edge set $E(G)$. Two vertices $u,v\in V(G)$ are \emph{adjacent} if and only if $uv \in E(G)$. We denote by $u\sim v$ if $u,v$ are adjacent and $u\nsim v$ if otherwise. The \emph{complement} $\overline{G}$ of $G$ is the graph with vertex set $V(G)$ and $uv \in \overline{G}$ if and only if $uv \notin G$.  

For a vertex $v\in V(G)$, the \emph{neighbourhood} of $v$, denoted by $N_G(v)$, is the set of all vertices adjacent to $v$. It can be simplified to $N(v)$ when there is no danger of ambiguity. For two disjoint vertex sets $A$ and $B$, we let $[A, B]$ denote the set of edges between $A$ and $B$. We say that $A$ is \emph{complete} (resp. \emph{anticomplete}) to $B$ if $|[A,B]|=|A||B|$ (resp. $[A,B]=\emptyset$). For a subset $S\subseteq V(G)$, we denote by $G-S$ the graph obtained from $G$ by deleting the vertices in $S$ together with their incident edges. We denote the subgraph induced on $X\subseteq V(G)$ by $G[X]$. We will, by a standard abuse of notation, identify a subset $X\subseteq V(G)$ with its induced subgraph $G[X]$ when the context is clear.

A \emph{matching} in $G$ is a set of edges without common vertices. The vertices incident to the edges of a matching $M$ are \emph{saturated} by $M$, and a \emph{perfect matching} is a matching that saturates every vertex of $G$. For bipartite graphs, Hall's theorem~\cite{HALL} gives a necessary and sufficient condition for the existence of a matching that saturates one part of the bipartition: if $G$ is a bipartite graph with bipartition $(X,Y)$, then $G$ has a matching that saturates $X$ if and only if $|N_G(S)|\geq|S|$ for every $S\subseteq X$. 

For two graphs $G$ and $H$, we say that $G$ \emph{contains} $H$ if $H$ is isomorphic to an induced subgraph of $G$. When $G$ does not contain $H$, we say that $G$ is \emph{H-free}. For a family $\mathcal{H}$ of graphs, we say that G is $\mathcal{H}$-free if $G$ is $H$-free for every $H\in\mathcal{H}$. For two vertex-disjoint graphs $G_1$ and $G_2$, the \emph{union} $G_1\cup G_2$ is the graph with vertex set $V(G_1\cup G_2)=V(G_1)\cup V(G_2)$ and edge set $E(G_1\cup G_2)=E(G_1)\cup E(G_2)$.

Let $P_n$ and $K_n$ be a \emph{path} and a \emph{complete graph} on $n$ vertices, respectively. The \emph{diamond} is a graph that consists of a $P_3$ with an additional vertex adjacent to all vertices of the $P_3$. 
A \emph{hole} is an induced cycle of length at least 5, and an \emph{antihole} is an induced subgraph whose complement graph is a hole. A hole or antihole is \emph{odd} or \emph{even} if it is of odd or even length, respectively.

\input{P2P4_Diamond.tpx}

We denote by $[k]$ the set $\{1,2,\ldots,k\}$ of the first $k$ positive integers. A \emph{$k$-colouring} of a graph $G$ is a mapping $c:V\rightarrow [k]$ such that $c(u)\neq c(v)$ whenever $u\sim v$. For a subset $A \subseteq V(G)$, we use $c(A)$ to denote the set of colours assigned to the vertices in $A$ under the colouring $c$. A graph $G$ is \emph{$k$-colourable} if it admits a $k$-colouring. The \emph{chromatic number} of $G$, denoted by $\chi(G)$, is the minimum number $k$ for which $G$ is $k$-colourable. We denote by $\omega(G)$ the \emph{clique number} of $G$ (the cardinality of its largest induced clique), and will simply write $\omega$ when the graph is clear from the context. Clearly, $\chi(G) \geq \omega(G)$. If $\chi(H) = \omega(H)$ for every induced subgraph $H$ of $G$, we say that $G$ is perfect. The famous Strong Perfect Graph Theorem, established by Chudnovsky, Robertson, Seymour and Thomas~\cite{SPGT}, states that a graph is perfect if and only if it is (odd hole, odd antihole)-free.

The notion of binding functions was introduced by Gy\'{a}rf\'{a}s~\cite{GY} in 1975, extending the concept of perfect graphs. Let $\mathcal{G}$ be a family of graphs. If there exists a real-valued function such that $\chi(G)\leq f(\omega(G))$ holds for every graph $G\in \mathcal{G}$, then we say that $\mathcal{G}$ is \emph{$\chi$-bounded}, and call $f$ a \emph{binding function} of $\mathcal{G}$. Gy\'{a}rf\'{a}s~\cite{GY} and Sumner~\cite{SUMNER} independently presented the following conjecture.

\begin{conjecture}\cite{GY}\cite{SUMNER}
For every forest $T$, the class of $T$-free graphs is $\chi$-bounded.
\end{conjecture}

Gy\'{a}rf\'{a}s~\cite{GY2} proved the conjecture for $T=P_t$: every $P_t$-free graph $G$ satisfied $\chi(G)\leq(t-1)^{\omega(G)-1}$. 
Note that this binding function is exponential; it is natural to ask whether this exponential bound can be improved to a polynomial bound for $P_t$-free graphs. This has proven to be a challenging problem, and little progress has been made thus far. Since $P_t$ is a connected graph and the Gy\'{a}rf\'{a}s-Sumner conjecture concerns forests, this naturally leads us to ask what the binding function is when forbidding a disconnected forest, such as $P_r \cup P_s$.

The problem of colouring $(P_r \cup P_s)$-free graphs (with $r,s \ge 2$) and their subclasses has been studied in various contexts in recent years (see~\cite{SURVEY} for a comprehensive survey). 
In particular, currently the best known binding function for $(P_2\cup P_3)$-free graphs is the same as that for $(P_2\cup P_4)$-free graphs: $f(\omega)=\frac{1}{6}\omega(\omega+1)(\omega+2)$~\cite{P2P3}. 
Although this function is polynomial, finding a better $\chi$-binding function, preferably linear and optimal, for subclasses of $(P_2\cup P_3)$-free and $(P_2\cup P_4)$-free graphs has received considerable attention. For example, Char and Karthick~\cite{gem} proved that $\chi(G)\leq \lceil \frac{5\omega(G)-1}{4} \rceil$ if $G$ is ($P_2\cup P_3$, gem)-free with $\omega(G)\geq4$. Li, Li and Wu~\cite{house} showed that if a graph $G$ is ($P_2\cup P_3$, house)-free with $\omega(G)\geq2$, then it satisfies the optimal bound $\chi(G)\leq2\omega(G)$. 

Here we are interested in the problem of obtaining optimal chromatic bounds for the class of ($P_r\cup P_s$, diamond)-free graphs. Note that the class of ($P_r\cup P_s$, diamond)-free graphs is a subclass of ($P_{r+s+1}$, diamond)-free graphs. For the latter, Schiermeyer and Randerath~\cite{SURVEY} proved the following general result:

\begin{lemma}\cite{SURVEY}\label{L}
Let $G$ be a connected ($P_t$, diamond)-free graph for $t\ge4$. Then $\chi(G)\le(t-2)(\omega(G)-1)$.

\end{lemma}

By Lemma \ref{L}, every ($P_r\cup P_s$, diamond)-free graph satisfies $\chi(G)\le(r+s-1)(\omega(G)-1)$. This bound is not expected to be tight in general. To date, the binding function of $(P_2 \cup P_3,\text{ diamond})$-free graphs has been completely determined and shown to be optimal:
Bharathi and Choudum~\cite{P2P3} established that $\chi(G)\le 4$ when $\omega(G)=2$, that $\chi(G)\le 6$ when $\omega(G)=3$, and that graphs in this class are perfect for $\omega(G)\ge 5$;
Karthick and Mishra~\cite{P6D2} showed that $\chi(G)\le 6$ for $\omega(G)\le 3$ and that this bound is tight when $\omega(G)=3$; Prashant, Francis and Raj~\cite{P2P3DIA} proved that $\chi(G)=4$ for $\omega(G)=4$.

As a superclass of ($P_2\cup P_3$, diamond)-free graphs, the class of ($P_2\cup P_4$, diamond)-free graphs was studied by Chen and Zhang~\cite{RANCHEN}. They established the following chromatic bounds: $\chi(G)\leq4$ for $\omega(G)=2$, $\chi(G)\leq7$ for $\omega(G)=3$, $\chi(G)\leq9$ for $\omega(G)=4$, and $\chi(G)\leq2\omega(G)+1$ for $\omega(G)\geq5$. They also proved that ($P_2\cup P_4$, diamond, $C_5$)-free graphs are perfect when $\omega(G)\geq5$. More recently, Angeliya, Karthick and Huang~\cite{P2P4DIA} made significant progress by demonstrating that for ($P_2\cup P_4$, diamond)-free graphs with $\omega(G)\geq3$, $\chi(G)\leq\max\{6,\omega(G)\}$. Their work also settled the case for $\omega(G)=4$ by proving $\chi(G)=4$. However, the bound was not tight for $\omega(G)=5$. Moreover, they improved on a result from~\cite{RANCHEN} by showing that if $G$ is ($P_2\cup P_4$, diamond, $C_5$)-free with $\omega(G)\geq4$, then $G$ is perfect. Subsequently, Chen, Wu and Zhang~\cite{RANCHEN2} showed that ($P_2\cup P_4$, bull, diamond)-free graphs satisfy $\chi(G)\leq\max\{4, \omega(G)\}$.

\subsection{Contribution}

This paper establishes a complete characterisation of the $\chi$-binding function for 
$(P_2 \cup P_4,\text{ diamond})$-free graphs.
It shows that the optimal $\chi$-binding function for this class is the same as that for 
$(P_2\cup P_3,\text{ diamond})$-free graphs~\cite{P2P3,P6D2,P2P3DIA},
and strengthens the chromatic bounds previously obtained for
$(P_2\cup P_4,\text{ diamond})$-free graphs by
Angeliya, Karthick and Huang~\cite{P2P4DIA}
and by Chen and Zhang~\cite{RANCHEN}.
Moreover, the proof relies on a different structural decomposition approach.

\begin{theorem}\label{THM1}
Let $G$ be a ($P_2\cup P_4$, diamond)-free graph. Then
\[
\chi(G) 
\begin{cases}
\le4, & \text{if } \omega(G) = 2, \\
\le6, & \text{if } \omega(G) = 3, \\
=\omega(G), & \text{if } \omega(G) \geq 4.
\end{cases}
\]
Furthermore, each of these bounds is tight.
\end{theorem}

It is important to note that, however, a ($P_2\cup P_4$, diamond)-free graph with clique number at least four is not necessarily perfect. To see this, consider the graph $H_n$ ($n\geq4$) formed by identifying an edge of a $K_n$ with an edge of a $C_5$ (Figure 2 illustrates this construction for $n=5$). One can verify that $H_n$ is ($P_2\cup P_4$, diamond)-free. Clearly, $\chi(H_n)=\omega(H_n)=n$. By the Strong Perfect Graph Theorem~\cite{SPGT}, $H_n$ is imperfect since it contains a $C_5$. 

\begin{figure}
    \centering

\begin{tikzpicture}[scale=0.8, line cap=round, line join=round]
  \tikzstyle{vtx}=[circle, fill=black, inner sep=2pt]

  \coordinate (a) at (-1,0);
  \coordinate (b) at (1,0);
  \coordinate (c) at (1.618,1.902);
  \coordinate (d) at (0,3.078);
  \coordinate (e) at (-1.618,1.902);
  \coordinate (h) at (3.236,3.078);  
  \coordinate (g) at (2.618,4.98);
  \coordinate (f) at (0.618,4.98);   
  \foreach \u/\v in {a/b,a/c,a/d,a/e,b/c,b/d,b/e,c/d,c/e,d/e}{
    \draw[line width=0.8pt] (\u)--(\v);
  }
  \draw[line width=0.8pt] (c)--(h)--(g)--(f)--(d);

  \foreach \x in {a,b,c,d,e,f,g,h}{
    \node[vtx] at (\x) {};
  }
\end{tikzpicture}
    \caption{An imperfect ($P_2 \cup P_4$, diamond)-free graph $G$ with $\chi(G)=\omega(G)=5$.}
    \label{fig:placeholder}
\end{figure}
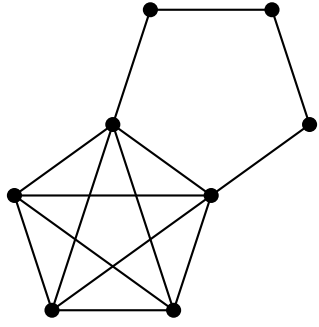

The remainder of this paper is organised as follows. Section 2 presents fundamental structural properties of ($P_2 \cup P_4$, diamond)-free graphs. This structural analysis leads to two primary cases, which are treated in detail in Sections 3 and 4, respectively. By synthesising these results we prove the main theorem in Section 5.

\medskip

We conclude this section with the following lemma, which will be used several times in the sequel.

\begin{lemma}\label{L0}
Let $G$ be a graph containing two disjoint cliques $X$ and $Y$ such that $\min\{|X|,|Y|\}\ge2$ and $\max\{|X|,|Y|\}\ge3$. If $[X,Y]$ induces a nonempty matching, then for any $x\in X$ and $y\in Y$, $G[X\cup Y]$ contains a $P_4$ whose vertex set includes both $x$ and $y$.
\end{lemma}
\begin{proof}
By symmetry, assume that $|X|\ge3$. Let $x\in X$ and $y\in Y$. We show that $G[X\cup Y]$ contains a $P_4$ whose vertex set contains both $x$ and $y$. 

First, consider the case when $x$ is anticomplete to $Y$. If $y$ is also anticomplete to $X$, then there exist vertices $x^{\prime}\in X$ and $y^{\prime}\in Y$ such that $x^{\prime}\sim y^{\prime}$. Hence, $\{x,x^{\prime},y^{\prime},y\}$ induces a $P_4$. If $y$ is not anticomplete to $X$, then there exists $x^{\prime}\in X$ such that $y\sim x^{\prime}$. In this case, $\{x,x^{\prime},y,y^{\prime\prime}\}$ induces a $P_4$, where $y^{\prime\prime}$ is any vertex in $Y\setminus\{y\}$.

Next, consider the case when $x$ is not anticomplete to $Y$. 
If $x\sim y$, then there exist vertices $x^{\prime}\in X$ and $y^{\prime}\in Y$ such that $x^{\prime}\nsim y^{\prime}$, since $|X|\ge3$. 
Hence, $\{x^{\prime},x,y,y^{\prime}\}$ induces a $P_4$. If $x\nsim y$, then there exists $y^{\prime}\in Y$ such that $x\sim y^{\prime}$. Since $|X|\ge3$, we can choose $x^{\prime}\in X$ such that $y\nsim x^{\prime}$. Thus, $\{x^{\prime},x,y^{\prime},y\}$ induces a $P_4$.
\end{proof}

\section{The Structure of ($P_2 \cup P_4$, diamond)-Free Graphs}

Let $G$ be a $(P_2\cup P_4,\text{ diamond})$-free graph with $\omega(G)\ge 2$. 
Let $A=\{a_1,a_2,\ldots,a_\omega\}$ be a maximum clique in $G$, and let $H=G-A$, which we assume is nonempty. 
For every $v\in V(H)$ we have $|N_A(v)|\le 1$: 
if $\omega\ge 3$ this follows from the fact that $G$ is diamond-free, 
and the case $\omega=2$ is trivial. 
We now partition $V(H)$ into the following subsets, where $i\in[\omega]$:
\begin{align*}
    C_i &=\{v\in V(H) \mid N_A(v)=\{a_i\}\},\\
    C_0 &=\{v\in V(H) \mid N_A(v)=\emptyset\}.
\end{align*}

\begin{lemma}\label{LD1}
    For all $1\leq i\neq j\leq \omega$, every subgraph induced by a subset of 
    $V(H)\setminus (C_i\cup C_j)$ is perfect.
\end{lemma}
\begin{proof}
This follows directly from the facts that $G$ is $(P_2\cup P_4)$-free and that every $P_4$-free graph is perfect~\cite{P4}.
\end{proof}

For the case $\omega =2$, the following lemma shows that $\chi(G)\leq 4$.

\begin{lemma}\label{LP2P4}
If $G$ is a ($P_2\cup P_4$, diamond)-free graph with $\omega =2$, then $\chi(G)\leq 4$.
\end{lemma}
\begin{proof}
Since $\omega=2$, both $C_1$ and $C_2$ are stable. 
By Lemma~\ref{LD1}, $G[C_0]$ is perfect, and hence it can be coloured with at most two colours. 
Now we may colour $a_1$ and $a_2$ with $1$ and $2$, respectively, and colour $C_1$ and $C_2$ with $2$ and $1$, respectively; 
finally colour $C_0$ with at most two colours from $\{3,4\}$.
\end{proof}

By Lemma~\ref{LP2P4}, we may assume that $\omega(G)\ge 3$ throughout this paper. 
\smallskip

Let $B=\{b_1,b_2,\ldots,b_k\}$ be a maximum clique in $H$, where $1\leq \omega(B)=k\leq \omega$. For $\omega \geq 4$, the next lemma provides an upper bound on the chromatic number of $G$ in terms of $\omega$ and $k$.

\begin{lemma}\label{LD2}
    If $G$ is a ($P_2\cup P_4$, diamond)-free graph with $\omega \geq 4$, then $\chi(G)\leq \max\{2k, \omega\}$.
\end{lemma}
\begin{proof}
Let 
\[
X=C_0\cup C_1\cup \cdots\cup C_{\lfloor\frac{\omega}{2}\rfloor}
\quad\text{and}\quad
Y=C_{\lfloor\frac{\omega}{2}\rfloor+1}\cup \cdots \cup C_\omega.
\]
By Lemma~\ref{LD1}, both $G[X]$ and $G[Y]$ are perfect, and hence each of $X$ and $Y$ can be 
coloured with at most $k$ colours. Moreover, $V(G)=A\cup X\cup Y$.

Now we may colour $a_i$ with $i$, and colour $X$ with colours from $\{\lfloor\frac{\omega}{2}\rfloor+1,\lfloor\frac{\omega}{2}\rfloor+2,\ldots,\lfloor\frac{\omega}{2}\rfloor+k\}$. If $k\leq \lfloor\frac{\omega}{2}\rfloor$, then we colour $Y$ with colours from $\{1,2,\ldots,\lfloor\frac{\omega}{2}\rfloor\}$. 
If  $k> \lfloor\frac{\omega}{2}\rfloor$, then we colour $Y$ with colours from $\{1,2,\ldots,\lfloor\frac{\omega}{2}\rfloor,\lfloor\frac{\omega}{2}\rfloor+k+1,\ldots,2k\}$. In both cases we obtain $\chi(G)\leq \max\{2k, \omega\}$.
\end{proof}

For $\omega =3$ and $k\leq 2$, next lemma shows that $\chi (G)\leq 6$.

\begin{lemma}\label{LD21}
If $G$ is a ($P_2\cup P_4$, diamond)-free graph with $\omega =3$ and $k\leq 2$, then $\chi(G)\leq 6$.
\end{lemma}

\begin{proof}
By Lemma~\ref{LD1}, each of the sets $C_0\cup C_1$, $C_2$, and $C_3$ induces a perfect graph. Since $k\le 2$, each of them can be coloured with at most two colours.  Hence, we may colour $a_i$ with colour $i$ for $i\in [3]$, colour $C_0\cup C_1$ with colours from $\{2,3\}$, colour $C_2$ with colours from $\{1,4\}$, and colour $C_3$ with colours from $\{5,6\}$.  
\end{proof}

By Lemmas~\ref{LD2} and~\ref{LD21}, if $\omega=3$ and $k\le 2$, then $\chi(G)\le 6$; 
whereas if $\omega\ge 4$ and $k\le 2$, then $\chi(G)=\omega$. 
Hence, for the remainder of this paper we assume that $k\ge 3$.

Since $H$ is diamond-free and $B$ is a maximum clique of $H$, and $k\ge 3$, 
for every $v\in V(H)\setminus B$ we have $|N_B(v)|\le 1$.  We partition $V(H)\setminus B$ into the following subsets, where $i\in [\omega]$ and $j\in [k]$:
\begin{align*}
    R_j &=\{v\in V(H)\setminus B \mid N_A(v)=\emptyset, N_B(v)= \{b_j\} \},\\
    S_i^j &=\{v\in V(H)\setminus B \mid N_A(v)=\{a_i\}, N_B(v)=\{b_j\}\},\\  
    T_i &=\{v\in V(H)\setminus B \mid N_B(v)=\emptyset, N_A(v)=\{a_i\}\},\\
    Z &=\{v\in V(H)\setminus B \mid N_A(v)=\emptyset,  N_B(v)=\emptyset \}.
\end{align*}
Note that $\omega(R_j), \omega(T_i) \le \omega - 1$ and $\omega(Z) \le k$.  Furthermore, by Lemma~\ref{LD1}, each of the sets $R_j$, $S_i^j$, $T_i$, and $Z$ induces a perfect graph. We then define
\[
R = \bigcup_{j=1}^{k} R_j,\quad 
S = \bigcup_{j=1}^k\bigcup_{i=1}^{\omega} S_i^j,\quad 
T = \bigcup_{i=1}^{\omega} T_i.
\]
Note that 
\[
V(G)=A\cup B\cup R\cup S\cup T\cup Z.
\]
If we define
\begin{align*}
    C_j^\prime &=\{v\in V(H)\setminus B \mid N_B(v)=\{b_j\}   \},\\
    C_0^\prime &=\{v\in V(H)\setminus B \mid N_B(v)=\emptyset\},
\end{align*}
then Table 1 illustrates the relations among the sets in the  partition introduced earlier.

\begin{table}[h]
\centering
\scalebox{0.9} 
{
\renewcommand{\arraystretch}{2} 
\setlength{\tabcolsep}{15pt} 
\begin{tabular}{c|cccccc}
\rule{0pt}{3ex} 
$\cap$ & $C_0'$ & $C_1'$ & $C_2'$ & $C_3'$ & $\cdots$ & $C_k'$ \\
\hline
\rule{0pt}{3ex}
$C_0$ & $Z$ & $R_1$ & $R_2$ & $R_3$ & $\cdots$ & $R_k$ \\
$C_1$ & $T_1$ & $S_1^1$ & $S_1^2$ & $S_1^3$ & $\cdots$ & $S_1^k$ \\
$C_2$ & $T_2$ & $S_2^1$ & $\ddots$ & & & $\vdots$ \\
$C_3$ & $T_3$ & $S_3^1$ & & $\ddots$ & & $\vdots$ \\
$\vdots$ & $\vdots$ & $\vdots$ & & & $\ddots$ & $\vdots$ \\
$C_\omega$ & $T_\omega$ & $S_\omega^1$ & $\cdots$ & $\cdots$ & $\cdots$ & $S_\omega^k$ \\
\end{tabular}
}
\caption{Partitions of $V(H)\setminus B$}
\label{tab: Partition}
\end{table}

With this partition, we now conclude Section 2 by presenting the following two claims.

\begin{claim}\label{CD1}
If there exists a vertex $v\in A$ such that $|N_B(v)|\geq 2$, then $\chi(G) = \omega$.
\end{claim}

\begin{proof}

Without loss of generality, let $a_1\in A$ be a vertex such that $|N_B(a_1)|\geq 2$.
Since $G$ is diamond-free, $a_1$ is complete to $B$; that is, $B\subseteq C_1$. Thus, $\{a_2,a_3\ldots,a_\omega \}$ is anticomplete to $\{b_1,b_2,\ldots,b_k\}$. It follows that $k\leq \omega-1$. Otherwise, $\{a_1\}\cup B$ would be a clique on $\omega+1$ vertices, a contradiction. Recall that we have already assumed $k\ge 3$ in the previous discussion; hence, $\omega\ge 4$.  We consider the following partition of $G$:

\[
V(G)=A\cup B\cup R\cup S\cup (T\setminus T_1)\cup (Z\cup T_1).
\]
Next, we list several properties of the sets $R$, $S$, $(T\setminus T_1)$, and $(Z\cup T_1)$.

\medskip

\noindent(1) $S_1^j=\emptyset$ for all $j\in [k]$, and $S_i^j$ is stable for all $i\in [\omega]\setminus\{1\}$ and $j\in [k]$.

This follows directly from the fact that $G$ is diamond-free.
\smallskip

\noindent(2) At most one set among $S_i^1,S_i^2,\ldots,S_i^k$ is nonempty, and at most one set among $S_2^j,S_3^j,\ldots,S_\omega^j$ is nonempty, for all $i\in [\omega]\setminus\{1\}$ and $j\in [k]$.

The two cases admit similar proofs; we therefore prove only the first, that is, at most one set among $S_i^1,S_i^2,\ldots,S_i^k$ is nonempty. To the contrary, suppose without loss of generality that there exist $x\in S_2^1$ and $y\in S_2^2$. If $x\sim y$, then $\{a_3,a_4,b_3,b_1,x,y\}$ induces a $P_2\cup P_4$. If $x\nsim y$, then $\{a_3,a_4,x,b_1,b_2,y\}$ induces a $P_2\cup P_4$. Both cases lead to a contradiction. 
\smallskip

\noindent(3) $Z\cup T_1$ is anticomplete to $R\cup S\cup (T\setminus T_1)$.

Suppose, to the contrary, that there exist vertices $x\in Z\cup T_1$ and $y\in R\cup S\cup (T\setminus T_1)$ such that $x\sim y$. If $y\in R\cup S$, then, by symmetry, we may assume that $y\sim b_1$. 
Since $\omega\ge 4$, there exist $a_i,a_j\in A\setminus\{a_1\}$ 
that are anticomplete to $x$ and $y$. Thus, $\{a_i,a_j,x,y,b_1,b_2\}$ induces a $P_2\cup P_4$. If $y\in T\setminus T_1$, then, by symmetry, we may assume that $y\sim a_2$. Since $k\geq 3$, there exist $b_p,b_q\in B$ that are anticomplete to $x$ and $y$. Hence, $\{b_p,b_q,x,y,a_2,a_3\}$ induces a $P_2\cup P_4$. In both cases, we obtain a contradiction.
\smallskip

\noindent(4) $R$ is anticomplete to $S\cup (T\setminus T_1)\cup (Z\cup T_1)$.

By (3) and symmetry, it suffices to show that $R_1$ is anticomplete to $S\cup (T\setminus T_1)$. Suppose to the contrary that there exist vertices $x\in R_1$ and $y\in S\cup (T\setminus T_1)$ such that $x\sim y$. For the case $y\in S$, if $y\nsim b_1$, then  
since $k\ge 3$, there exists $b_t\in B\setminus \{b_1\}$ nonadjacent to $y$; 
and since $\omega\ge 4$, there exist $a_i,a_j\in A\setminus\{a_1\}$ 
anticomplete to $x$ and $y$.  
Thus, $\{a_i,a_j,b_t,b_1,x,y\}$ induces a $P_2\cup P_4$. 
If $y\sim b_1$ and $y\sim a_p$, then $b_2$ and $b_3$ are anticomplete to $x$ and $y$. Hence, $\{b_2,b_3,x,y,a_p,a_q\}$ induces a $P_2\cup P_4$, where $a_q\in A\setminus\{a_1,a_p\}$. For the case $y\in T\setminus T_1$, since $\omega\geq 4$, there exist $a_i,a_j\in A\setminus \{a_1\}$ that are anticomplete to $x$ and $y$. Thus, $\{a_i,a_j,b_2,b_1,x,y\}$ induces a $P_2\cup P_4$. All cases lead to a contradiction.
\smallskip

\noindent(5) $R_i$ is anticomplete to $R_j$ for distinct $i,j\in [k]$.

For the sake of contradiction, suppose by symmetry that there exist $x\in R_1$ and $y\in R_2$ such that $x\sim y$. Note that $\{a_2,a_3,b_3,b_1,x,y\}$ induces a $P_2\cup P_4$, a contradiction.
\smallskip

\noindent(6) $T\setminus T_1$ is anticomplete to $R\cup S\cup (Z\cup T_1)$
\smallskip

\noindent(7) $T_i$ is anticomplete to $T_j$ for distinct $i,j\in [\omega]\setminus \{1\}$.

The proofs of (6) and (7) are analogous to those of (4) and (5), and hence are omitted.
\smallskip

\noindent(8) $S$ is anticomplete to $R\cup (T\setminus T_1)\cup (Z\cup T_1)$.

This follows from (3), (4), and (6).
\medskip

By (1) and (2) and by symmetry, we may let $S=\bigcup_{i=1}^k S_{i+1}^i$. 
By Lemma~\ref{LD1}, each of the sets $R_j$, $T_i$, and $Z\cup T_1$ induces a perfect graph. 
Recall that $k \le \omega - 1$. 
Hence, each of these sets can be coloured with at most $\omega - 1$ colours. 
Finally, we colour $G$ with $\omega$ colours as follows.

\begin{itemize}
\item Colour $a_i$ with colour $i$ for $i\in [\omega]$.
\item Colour $b_i$ with colour $i+1$ for $i\in [k]$.
\item By (8), colour $S_{i+1}^i$ with colour $i$ for $i\in [k]$ (recall that $S_{i+1}^i$ is stable).
\item By (3), colour $Z\cup T_1$ with at most $\omega-1$  colours from $[\omega]\setminus \{1\}$.
\item By (4) and (5), colour $R_i$ with at most $\omega-1$ colours from $[\omega]\setminus \{i+1\}$ for $i\in [k]$.
\item By (6) and (7), colour $T_i$ with at most $\omega-1$ colours from $[\omega]\setminus \{i\}$ for $i\in [\omega]\setminus \{1\}$.
\end{itemize}

This completes the proof of Claim~\ref{CD1}.
\end{proof}

By Claim~\ref{CD1}, we may assume that, for every maximum clique $B$
of $H$, every vertex of $A$ has at most one neighbour in $B$;
otherwise, Claim~\ref{CD1}, applied to such a maximum clique, gives
$\chi(G)=\omega$. Since every vertex of $H$ has at most one neighbour
in $A$, the following holds for every maximum clique
$B=\{b_1,b_2,\ldots,b_k\}$ of $H$:
\begin{equation}\label{E1}
    \text{$[\{a_1,a_2,\ldots,a_\omega\},
    \{b_1,b_2,\ldots,b_k\}]$ is a matching in $G$.}
    \tag{$\ast$}
\end{equation}

\begin{claim}\label{C2}
If $H$ contains a $K_k$ whose vertex set is anticomplete to $A$, then $\chi(G)=\omega$.
\end{claim}
\begin{proof}
Let $B$ be the $K_k$ that is anticomplete to $A$. Since $G$ is diamond-free, every $S_i^j$ is stable. We first list some properties of $Z$, $R$, $S$, and $T$. The proofs of these properties are analogous to those given in the proof of Claim~\ref{CD1}, and hence are omitted.
\medskip

\noindent(1) At most one set among $S_i^1,S_i^2,\ldots,S_i^k$ is nonempty, and at most one set among $S_1^j,S_2^j,\ldots,S_\omega^j$ is nonempty, for all $i\in [\omega]$ and $j\in [k]$.
\smallskip

\noindent(2) $Z$ is anticomplete to $R\cup S\cup T$.
\smallskip

\noindent(3) $R$ is anticomplete to $S\cup T\cup Z$.
\smallskip

\noindent(4) $R_i$ is anticomplete to $R_j$ for distinct $i,j\in [k]$.
\smallskip

\noindent(5) $T$ is anticomplete to $R\cup S\cup Z$
\smallskip

\noindent(6) $T_i$ is anticomplete to $T_j$ for distinct $i,j\in [\omega]$.
\smallskip

\noindent(7) $S$ is anticomplete to $R\cup T\cup Z$.
\medskip

By (1) and by symmetry, we may let $S=\bigcup_{i=1}^k S_{i}^i$. By Lemma~\ref{LD1}, each of the sets $R_j$, $T_i$, and $Z$ induces a perfect graph. Since $\omega(R_j), \omega(T_i) \leq \omega - 1$ and $\omega(Z) \le k\leq \omega$, we may colour them with at most $\omega - 1$, $\omega - 1$, and $\omega$ colours, respectively. We now present an $\omega$-colouring of $G$ as follows.

\begin{itemize}
\item Colour $a_i$ with colour $i$ for $i\in [\omega]$.
\item Colour $b_i$ with colour $i$ for $i\in [k]$.
\item By (7), if $k=1$, then colour $S_1^1$ with colour $2$.  
If $k\ge 2$, then colour $S_i^i$ with colour $i+1$ for $i\in [k-1]$, and colour $S_k^k$ with colour~1.
\item By (2), colour $Z$ with at most $\omega$  colours from $[\omega]$.
\item By (3) and (4), colour $R_i$ with at most $\omega-1$ colours from $[\omega]\setminus \{i\}$ for $i\in [k]$.
\item By (5) and (6), colour $T_i$ with at most $\omega-1$ colours from $[\omega]\setminus \{i\}$ for $i\in [\omega]$.
\end{itemize}

This completes the proof of Claim~\ref{C2}.
\end{proof}
By Claim~\ref{C2}, we may assume that
\begin{equation}\label{E2}
    \text{$H$ contains no $K_k$ whose vertex set is anticomplete to $A$.}
    \tag{$\ast\ast$}
\end{equation}

\section{The Case $k=3$} 

We continue with the notation and definitions introduced earlier. By Lemma~\ref{LD2}, if $\omega \ge 6$ and $k=3$, then $\chi(G)=\omega$.  
Hence, throughout this section we may assume that $\omega\in\{3,4,5\}$. 
The aim of this section is to prove the following three lemmas.

\begin{lemma}\label{L31}
If $G$ is a $(P_2\cup P_4,\text{ diamond})$-free graph with $\omega =3$ and $k=3$, then $\chi(G)\le 6$.  
\end{lemma}
\begin{lemma}\label{L32}
If $G$ is a $(P_2\cup P_4,\text{ diamond})$-free graph with $\omega =4$ and $k=3$, then $\chi(G)=4$. 
\end{lemma}
\begin{lemma}\label{L33}
If $G$ is a $(P_2\cup P_4,\text{ diamond})$-free graph with $\omega =5$ and $k=3$, then $\chi(G)=5$. 
\end{lemma}

\begin{proof}[\normalfont\textbf{Proof of Lemma~\ref{L31}}]

We first claim that if there exists $C_i$ among $C_1, C_2, C_3$ such that $\omega(C_0 \cup C_i) \le 2$, then $\chi(G) \le 6$. 
Without loss of generality, assume that $\omega(C_0 \cup C_1) \le 2$. 
Note that $\omega(C_2), \omega(C_3) \le 2$, and by Lemma~\ref{LD1}, each of $C_0 \cup C_1$, $C_2$, and $C_3$ induces a perfect graph. 
Thus, we may use at most two colours to colour each of $C_0 \cup C_1$, $C_2$, and $C_3$. 
Then we colour $a_1, a_2, a_3$ with colours 1, 2, and 3, respectively; 
colour $C_0 \cup C_1$ with colours from $\{2,3\}$, 
colour $C_2$ with colours from $\{1,4\}$, 
and colour $C_3$ with colours from $\{5,6\}$.

Hence, we may assume that for each $C_i$ among $C_1, C_2,$ and $C_3$, the subgraph $G[C_0 \cup C_i]$ contains a $K_3$. 
Since $k = 3$, we may regard $B$ as such a $K_3$. 
Then, by~\eqref{E1}, any $K_3$ in $G[C_0 \cup C_i]$ has at most one vertex in $C_i$. 
Moreover, by~\eqref{E2},  $G[C_0]$ is $K_3$-free. 
Therefore, every $K_3$ contained in $G[C_0 \cup C_i]$ consists of exactly one vertex in $C_i$ and two vertices in $C_0$.

Let $X \subseteq C_0 \cup C_1$ and $Y \subseteq C_0 \cup C_2$ be two vertex sets such that both $X$ and $Y$ induce a $K_3$. 
Let $X = \{x_0, x_0^\prime, x_1\}$ and $Y = \{y_0, y_0^\prime, y_2\}$, where $\{x_0, x_0^\prime, y_0, y_0^\prime\} \subseteq C_0$, $x_1 \in C_1$, and $y_2 \in C_2$. Since $k = 3$ and $G$ is diamond-free, $|X\cap Y|\leq 1$.

If $|X \cap Y| = 1$, then we may assume that $x_0 = y_0$ (see Figure~3(a)). 
Since $H$ is both $K_4$-free and diamond-free, $\{x_0^\prime, x_1\}$ is anticomplete to $\{y_0^\prime, y_2\}$. 
Hence, $\{x_0^\prime, x_1, y_0^\prime, y_2, a_2, a_3\}$ induces a $P_2 \cup P_4$, a contradiction.

If $|X \cap Y| = 0$ (see Figure~3(b)), then since $k = 3$ and $G$ is diamond-free, $[\{x_0, x_0^\prime\},Y]$ is a matching in $G$. 
If $\{x_0,x_0^\prime\}$ is anticomplete to $Y$, then $\{x_0,x_0^\prime,y_0,y_2,a_2,a_3\}$ induces a $P_2\cup P_4$, a contradiction. Hence, the matching $[\{x_0,x_0^\prime\},Y]$ is nonempty.
By Lemma~\ref{L0},  $G[\{x_0, x_0^\prime\} \cup Y]$ contains a $P_4$; the vertices of this $P_4$, together with $\{a_1, a_3\}$, induce a $P_2 \cup P_4$, a contradiction.

Thus, at least one of $G[C_0 \cup C_1]$ and $G[C_0 \cup C_2]$ is $K_3$-free, and hence $\chi(G) \le 6$.
\end{proof}

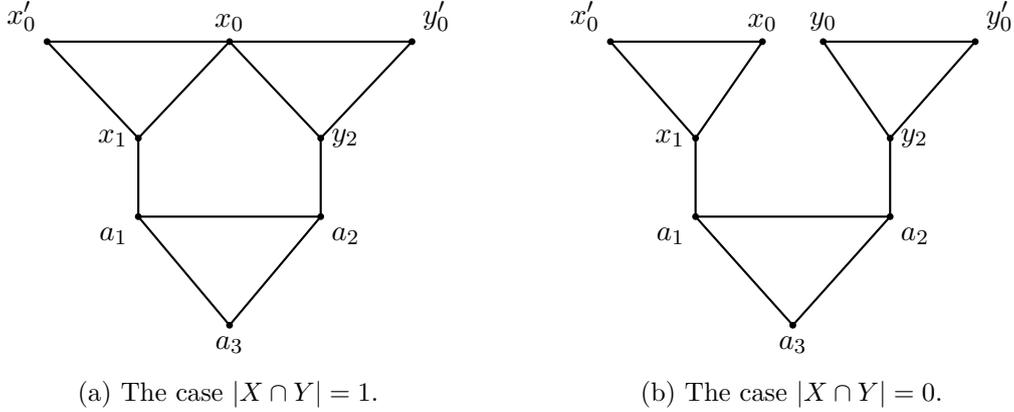
\begin{figure}[htbp]
\centering

\begin{subfigure}{0.44\linewidth}
\centering
\begin{tikzpicture}[scale=0.8, line width=0.8pt, every node/.style={font=\normalsize}]
  
  \coordinate (LX0p) at (-3, 2.5); 
  \coordinate (LX0)  at (0, 2.5);   
  \coordinate (LY0p) at (3, 2.5);      
  \coordinate (LX1)  at (-1.5, 0.9); 
  \coordinate (LY1)  at (1.5, 0.9);  
  \coordinate (LA1)  at (-1.5, -0.4);
  \coordinate (LA2)  at (1.5, -0.4); 
  \coordinate (LA3)  at (0, -2.2);   

  \draw (LX0p) -- (LX0) -- (LY0p);
  \draw (LX0p) -- (LX1) -- (LX0);
  \draw (LX0)  -- (LY1) -- (LY0p);
  \draw (LX1)  -- (LA1);
  \draw (LY1)  -- (LA2);
  \draw (LA1)  -- (LA2);
  \draw (LA1)  -- (LA3) -- (LA2);

  \foreach \v in {LX0p,LX0,LY0p,LX1,LY1,LA1,LA2,LA3}{\fill (\v) circle (1.6pt);}

  \node[above left]  at (LX0p) {$x_0'$};
  \node[above]       at (LX0)  {$x_0$};
  \node[above right] at (LY0p) {$y_0'$};
  \node[left]        at (LX1)  {$x_1$};
  \node[right]       at (LY1)  {$y_2$};
  \node[below left]  at (LA1)  {$a_1$};
  \node[below right] at (LA2)  {$a_2$};
  \node[below]       at (LA3)  {$a_3$};

\end{tikzpicture}
\caption{The case $|X\cap Y|=1$.}
\label{}
\end{subfigure}
\hspace{0.3em}
\begin{subfigure}{0.44\linewidth}
\centering
\begin{tikzpicture}[scale=0.8, line width=0.8pt, every node/.style={font=\normalsize}]
 
  \coordinate (RX0p) at (-3.0, 2.5); 
  \coordinate (RX0)  at (-0.5, 2.5); 
  \coordinate (RY0)  at (0.5, 2.5);   
  \coordinate (RY0p) at (3.0, 2.5);   

  \coordinate (RX1)  at (-1.6, 0.9);  
  \coordinate (RY2)  at ( 1.6, 0.9);  

  \coordinate (RA1)  at (-1.6, -0.4);
  \coordinate (RA2)  at ( 1.6, -0.4);
  \coordinate (RA3)  at ( 0.0, -2.2);

  \draw (RX0p) -- (RX0);
  \draw (RY0)  -- (RY0p);
  \draw (RX0p) -- (RX1) -- (RX0);
  \draw (RY0)  -- (RY2) -- (RY0p);
  \draw (RX1) -- (RA1);
  \draw (RY2) -- (RA2);
  \draw (RA1) -- (RA2);
  \draw (RA1) -- (RA3) -- (RA2);

  \foreach \v in {RX0p,RX0,RY0,RY0p,RX1,RY2,RA1,RA2,RA3}{\fill (\v) circle (1.6pt);}

  \node[above left]  at (RX0p) {$x_0'$};
  \node[above]       at (RX0)  {$x_0$};
  \node[above]       at (RY0)  {$y_0$};
  \node[above right] at (RY0p) {$y_0'$};

  \node[left]        at (RX1)  {$x_1$};
  \node[right]       at (RY2)  {$y_2$};

  \node[below left]  at (RA1)  {$a_1$};
  \node[below right] at (RA2)  {$a_2$};
  \node[below]       at (RA3)  {$a_3$};

\end{tikzpicture}
\caption{The case $|X\cap Y|=0$.}
\label{}
\end{subfigure}

\caption{Two cases in the proof of Lemma~\ref{L31}.}
\label{}
\end{figure}

\begin{proof}[\normalfont\textbf{Proof of Lemma~\ref{L32}}]
We first claim that if there exist $C_i, C_j$ among $C_1$, $C_2$, $C_3$, and $C_4$ such that $\omega(C_0 \cup C_i \cup C_j) \le 2$, then $\chi(G) = 4$. 
Without loss of generality, assume that $\omega(C_0 \cup C_1 \cup C_2) \le 2$. 
Since $k=3$, if $G[C_3\cup C_4]$ contains a $K_3$, then we may take this $K_3$ as $B$, which contradicts~\eqref{E1}. 
Hence, $G[C_3 \cup C_4]$ is $K_3$-free, and thus $\omega(C_3 \cup C_4) \le 2$. 
We may then colour $a_i$ with colour $i$ for $i \in [4]$, 
colour $C_0 \cup C_1 \cup C_2$ with colours from $\{3,4\}$, 
and colour $C_3 \cup C_4$ with colours from $\{1,2\}$.

Thus, we may assume that for any $C_i, C_j$ among $C_1$, $C_2$, $C_3$, and $C_4$, 
the subgraph $G[C_0 \cup C_i \cup C_j]$ contains a $K_3$. 
Since $k = 3$, we may regard $B$ as such a $K_3$. 
By~\eqref{E1}, any $K_3$ in $G[C_0 \cup C_i \cup C_j]$ has at most one vertex in $C_i$ 
and at most one vertex in $C_j$. 
Moreover, by~\eqref{E2}, the subgraph $G[C_0]$ is $K_3$-free. 
Therefore, every $K_3$ contained in $G[C_0 \cup C_i \cup C_j]$ contains at least one vertex in $C_0$ and at most two vertices in $C_0$. Let $X \subseteq C_0 \cup C_1 \cup C_2$ be a vertex set such that $G[X]$ is a $K_3$. 

If $|X \cap C_0| = 2$, then by symmetry we may assume that $|X \cap C_1| = 1$ 
(the case $|X \cap C_2| = 1$ is analogous). Let $X=\{x_0,x_0^\prime,x_1\}$, where $x_0,x_0^\prime \in C_0$ and $x_1\in C_1$. Now consider a vertex set $Y\subseteq C_0\cup C_3\cup C_4$ that induces a $K_3$. By symmetry, we may assume that $|Y\cap C_4|=1$ (the case $|Y \cap C_3| = 1$ is analogous). 
Let $Y=\{y_0,y_4,y\}$, where $y_0\in C_0$, $y_4\in C_4$, and $y\in C_0$ or $C_3$. Since $k=3$ and  $G$ is diamond-free, we have $|X\cap Y|\leq 1$. 

For the case $|X\cap Y|=1$, we may assume that $x_0=y_0$ (see Figure~4(a)). Since $H$ is $K_4$-free and diamond-free, $\{x_0^\prime,x_1\}$ is anticomplete to $\{y,y_4\}$. Thus, $\{x_0^\prime,x_1,y,y_4,a_4,a_2\}$ induces a $P_2\cup P_4$, a contradiction. For the case $|X\cap Y|=0$ (see Figure~4(b)), since $k=3$ and $G$ is diamond-free, $[\{x_0,x_0^\prime\},Y]$ is a matching in $G$. If $\{x_0,x_0^\prime\}$ is anticomplete to $Y$, then $\{x_0,x_0^\prime,y_0,y_4,a_4,a_2\}$ induces a $P_2\cup P_4$, a contradiction. Hence, $[\{x_0,x_0^\prime\},Y]$ is a nonempty matching in $G$. By Lemma~\ref{L0}, $G[\{x_0,x_0^\prime\}\cup Y]$ contains a $P_4$. The vertices of this $P_4$, together with $\{a_1, a_2\}$, induce a $P_2 \cup P_4$, a contradiction.

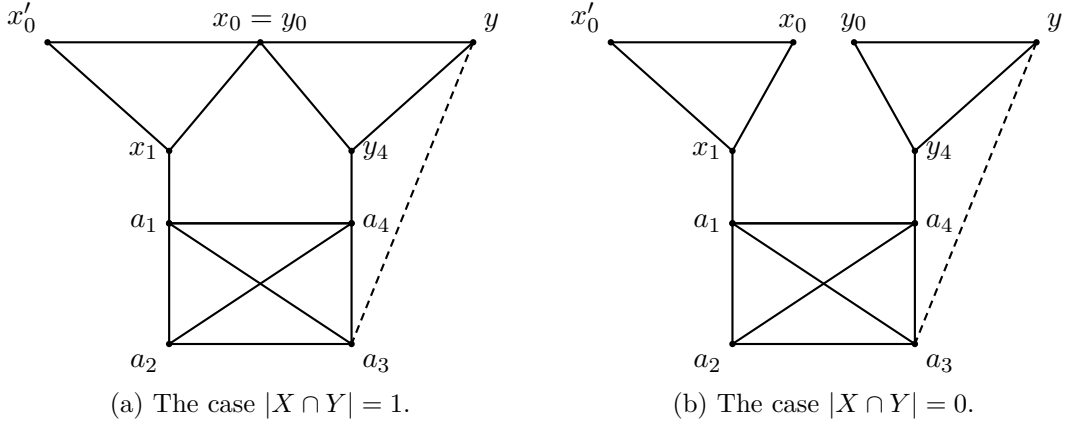
\begin{figure}[htbp]
\centering

\begin{subfigure}{0.44\linewidth}
\centering
\begin{tikzpicture}[scale=0.8, line width=0.8pt, every node/.style={font=\normalsize}]

\path[use as bounding box] (0.5,0.5) rectangle (8.5,6.5);

  \coordinate (LX0p) at (1, 6);
  \coordinate (LXY0) at (4.5,  6);
  \coordinate (LY)   at (8,  6);
  \coordinate (LX1)  at (3, 4.2);
  \coordinate (LY4)  at (6, 4.2);
  \coordinate (LMidL) at (3, 3);
  \coordinate (LMidR) at (6, 3);
  \coordinate (LA2) at (3, 1);
  \coordinate (LA3) at (6, 1);
  \coordinate (LA1) at (3,  3);
  \coordinate (LA4) at (6,  3);

  \draw (LX0p) -- (LXY0) -- (LY);
  \draw (LX0p) -- (LX1) -- (LXY0);
  \draw (LXY0) -- (LY4) -- (LY);
  \draw (LX1) -- (LMidL);
  \draw (LY4) -- (LMidR);
  \draw (LMidL) -- (LMidR);
  \draw (LA2) -- (LA1) -- (LA4) -- (LA3) -- (LA2);
  \draw (LA1) -- (LA3);
  \draw (LA2) -- (LA4);
  
  \draw[densely dashed] (LY) -- (LA3);

  \foreach \v in {LX0p,LXY0,LY,LX1,LY4,LMidL,LMidR,LA1,LA2,LA3,LA4}{
    \fill (\v) circle (1.5pt);
  }

  \node[above left]  at (LX0p) {$x_0'$};
  \node[above]       at (LXY0) {$x_0 = y_0$};
  \node[above right] at (LY)   {$y$};
  \node[left]        at (LX1)  {$x_1$};
  \node[right]       at (LY4)  {$y_4$};
  \node[left]        at (LA1)  {$a_1$};
  \node[right]       at (LA4)  {$a_4$};
  \node[below left]  at (LA2)  {$a_2$};
  \node[below right] at (LA3)  {$a_3$};
\end{tikzpicture}

\caption{The case $|X\cap Y|=1$.}
\label{}
\end{subfigure}
\hspace{0.3em}
\begin{subfigure}{0.44\linewidth}
\centering
\begin{tikzpicture}[scale=0.8, line width=0.8pt, every node/.style={font=\normalsize}]

\path[use as bounding box] (0.5,0.5) rectangle (8.5,6.5);

    \coordinate (RX0p) at (1, 6);
    \coordinate (RX0)  at (4, 6);
    \coordinate (RX1)  at (3, 4.2);
    \coordinate (RY0)  at (5, 6);
    \coordinate (RY)   at (8, 6);
    \coordinate (RY4)  at (6, 4.2);
    \coordinate (RMidL) at (3, 3);
    \coordinate (RMidR) at (6, 3);
    \coordinate (RA2)   at (3, 1);
    \coordinate (RA3)   at ( 6, 1);
    \coordinate (RA1)   at (3,  3);
    \coordinate (RA4)   at (6,  3);

    \draw (RX0p) -- (RX0);
    \draw (RX0p) -- (RX1) -- (RX0);
    \draw (RY0) -- (RY);
    \draw (RY0) -- (RY4) -- (RY);
    \draw (RX1) -- (RMidL);
    \draw (RY4) -- (RMidR);
    \draw (RMidL) -- (RMidR);
    \draw (RA2) -- (RA1) -- (RA4) -- (RA3) -- (RA2);
    \draw (RA1) -- (RA3);
    \draw (RA2) -- (RA4);
    
    \draw[densely dashed] (RY) -- (RA3);

    \foreach \v in {RX0p,RX0,RX1,RY0,RY,RY4,RMidL,RMidR,RA1,RA2,RA3,RA4}{
      \fill (\v) circle (1.5pt);
    }

    \node[above left]  at (RX0p) {$x_0'$};
    \node[above]       at (RX0)  {$x_0$};
    \node[left]        at (RX1)  {$x_1$};
    \node[above]       at (RY0)  {$y_0$};
    \node[above right] at (RY)   {$y$};
    \node[right]       at (RY4)  {$y_4$};
    \node[left]        at (RA1)  {$a_1$};
    \node[right]       at (RA4)  {$a_4$};
    \node[below left]  at (RA2)  {$a_2$};
    \node[below right] at (RA3)  {$a_3$};
\end{tikzpicture}

\caption{The case $|X\cap Y|=0$.}
\label{}
\end{subfigure}

\caption{The case $|X\cap C_0|=2$ in the proof of Lemma~\ref{L32}. Here a dashed line indicates that the two vertices may or may not be adjacent.}
\label{}
\end{figure}

If $|X \cap C_0| = 1$, consider a vertex set $Y \subseteq C_0 \cup C_1 \cup C_4$ that induces a $K_3$. If $|Y \cap C_0| = 2$, then by the arguments in the preceding paragraphs, we deduce that $C_0 \cup C_2 \cup C_3$ is $K_3$-free, a contradiction. Hence, $|Y \cap C_0| = 1$. Let $X=\{x_0,x_1,x_2\}$ and $Y=\{y_0,y_1,y_4\}$, where $x_0,y_0\in C_0$, $x_1,y_1\in C_1$, $x_2\in C_2$, and $y_4\in C_4$. Since $k=3$ and  $G$ is diamond-free, we have $|X\cap Y|\leq 1$. 

For the case $|X \cap Y| = 1$, either $x_0 = y_0$ or $x_1 = y_1$. If $x_0=y_0$ (see Figure~5(a)), then $\{x_1,x_2\}$ is anticomplete to $\{y_1,y_4\}$ since $H$ is $K_4$-free and diamond-free. Thus, $\{x_1,x_2,y_1,y_4,a_4,a_3\}$ induces a $P_2\cup P_4$, a contradiction. If $x_1=y_1$ (see Figure~5(b)), then $\{x_0,x_2\}$ is anticomplete to $\{y_0,y_4\}$ since $H$ is $K_4$-free and diamond-free. Thus, $\{x_0,x_2,y_0,y_4,a_4,a_3\}$ induces a $P_2\cup P_4$, a contradiction. For the case $|X\cap Y|=0$ (see Figure~5(c)), since $k=3$ and $G$ is diamond-free, $[\{x_0,x_1\},Y]$ is a matching in $G$. 
If $\{x_0,x_1\}$ is anticomplete to $Y$, then $\{x_0,x_1,y_0,y_4,a_4,a_3\}$ induces a $P_2\cup P_4$, a contradiction. Hence, $[\{x_0,x_1\},Y]$ is a nonempty matching in $G$. By Lemma~\ref{L0}, $G[\{x_0,x_1\}\cup Y]$ contains a $P_4$. The vertices of this $P_4$, together with $\{a_2, a_3\}$, induce a $P_2 \cup P_4$, a contradiction.

\begin{figure}[htbp]
\centering

\begin{subfigure}{0.32\linewidth}
\centering
\resizebox{\linewidth}{!}{%
\begin{tikzpicture}[line width=0.8pt, every node/.style={font=\Large}]

\path[use as bounding box] (0.8,0.5) rectangle (8.2,6.5);

\coordinate (X2)  at (1.5,6);
\coordinate (X0)  at (4.5,6);
\coordinate (Y4)  at (7.5,6);
\coordinate (X1)  at (3,4.2);
\coordinate (Y1)  at (6,4.2);
\coordinate (A1)  at (3,3);
\coordinate (A4)  at (6,3);
\coordinate (A2)  at (3,1);
\coordinate (A3)  at (6,1);

\draw (X2) -- (X0) -- (Y4);
\draw (X2) -- (A2);
\draw (X1) -- (X2);
\draw (Y4) -- (A4);
\draw (X1) -- (X0) -- (Y1);
\draw (X1) -- (A1) -- (A2);
\draw (Y1) -- (Y4);
\draw (A1) -- (A4);
\draw (A2) -- (A4);
\draw (A1) -- (A3);
\draw (A1) -- (Y1);
\draw (A2) -- (A3);
\draw (A3) -- (A4);

\foreach \p in {X2,X0,Y4,X1,Y1,A1,A2,A3,A4}{
  \fill (\p) circle (2pt);
}

\node[above left]  at (X2) {$x_2$};
\node[above]       at (X0) {$x_0 = y_0$};
\node[above right] at (Y4) {$y_4$};
\node[left]        at (X1) {$x_1$};
\node[right]       at (Y1) {$y_1$};
\node[left]        at (A1) {$a_1$};
\node[below left]  at (A2) {$a_2$};
\node[below right] at (A3) {$a_3$};
\node[right]       at (A4) {$a_4$};

\end{tikzpicture}%
}
\caption{The case $x_0=y_0$.}
\end{subfigure}
\hspace{0.5em}
\begin{subfigure}{0.32\linewidth}
\centering
\resizebox{\linewidth}{!}{%
\begin{tikzpicture}[line width=0.8pt, every node/.style={font=\Large}]

\path[use as bounding box] (0.8,0.5) rectangle (8.2,6.5);

\coordinate (X2)  at (1.5,6);
\coordinate (X1)  at (4.5,6);
\coordinate (Y0)  at (7.5,6);
\coordinate (X0)  at (3,4.2);
\coordinate (Y4)  at (6,4.2);
\coordinate (A1)  at (3,3);
\coordinate (A4)  at (6,3);
\coordinate (A2)  at (3,1);
\coordinate (A3)  at (6,1);

\draw (X1) -- (A1);
\draw (X2) -- (X1) -- (Y0);
\draw (X2) -- (A2);
\draw (X0) -- (X2);
\draw (X0) -- (X1) -- (Y4);
\draw (A1) -- (A2);
\draw (Y4) -- (Y0);
\draw (A1) -- (A4);
\draw (A2) -- (A4);
\draw (A1) -- (A3);
\draw (A2) -- (A3);
\draw (A3) -- (A4);
\draw (A4) -- (Y4);

\foreach \p in {X2,X0,Y4,X1,Y0,A1,A2,A3,A4}{
  \fill (\p) circle (2pt);
}

\node[above left]  at (X2) {$x_2$};
\node[above]       at (X1) {$x_1 = y_1$};
\node[above right] at (Y0) {$y_0$};
\node[left]        at (X0) {$x_0$};
\node[right]       at (Y4) {$y_4$};
\node[left]        at (A1) {$a_1$};
\node[below left]  at (A2) {$a_2$};
\node[below right] at (A3) {$a_3$};
\node[right]       at (A4) {$a_4$};

\end{tikzpicture}%
}
\caption{The case $x_1=y_1$.}
\end{subfigure}%
\hspace{0.5em}%
\begin{subfigure}{0.32\linewidth}
\centering
\resizebox{\linewidth}{!}{%
\begin{tikzpicture}[line width=0.8pt, every node/.style={font=\Large}]

\path[use as bounding box] (0.8,0.5) rectangle (8.2,6.5);

\coordinate (A1) at (3,3); 
\coordinate (A4) at (6,3); 
\coordinate (A2) at (3,1); 
\coordinate (A3) at (6,1);

\coordinate (X2) at (1.5,6);      
\coordinate (X0) at (4,6);      
\coordinate (Y0) at (5,6);      
\coordinate (Y4) at (7.5,6);      

\coordinate (X1) at (3,4.2);    
\coordinate (Y1) at (6,4.2);    

\draw (X2) -- (X0);
\draw (Y0) -- (Y4);
\draw (X2) -- (X1);
\draw (X1) -- (X0);
\draw (X2) -- (A2);   
\draw (X1) -- (A1);   
\draw (Y0) -- (Y1);
\draw (Y1) -- (Y4);
\draw (Y4) -- (A4);   
\draw (A1) -- (A4);   
\draw (A2) -- (A3);   
\draw (A1) -- (A2);   
\draw (A4) -- (A3);   
\draw (A1) -- (A3);  
\draw (A2) -- (A4);  
\draw (A1) -- (Y1);

\foreach \p in {X2,X0,Y0,Y4,X1,Y1,A1,A2,A3,A4}{
\fill (\p) circle (2pt);
}

\node[above left]  at (X2) {$x_2$};
\node[above]       at (X0) {$x_0$};
\node[above]       at (Y0) {$y_0$};
\node[above right] at (Y4) {$y_4$};
\node[left]        at (X1) {$x_1$};
\node[right]       at (Y1) {$y_1$};
\node[left]        at (A1) {$a_1$};
\node[below left]  at (A2) {$a_2$};
\node[below right] at (A3) {$a_3$};
\node[right]       at (A4) {$a_4$};

\end{tikzpicture}%
}
\caption{The case $|X\cap Y|=0$.}
\end{subfigure}

\caption{The case $|X\cap C_0|=1$ in the proof of Lemma~\ref{L32}.}
\label{}
\end{figure}
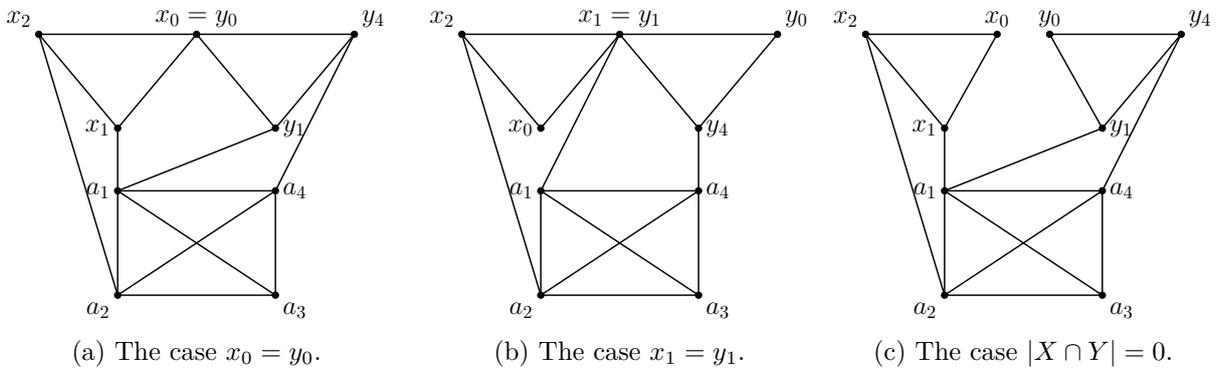

Thus, at least one of $G[C_0 \cup C_1\cup C_2]$, $G[C_0 \cup C_3\cup C_4]$, and $G[C_0 \cup C_1\cup C_4]$ is $K_3$-free, and hence $\chi(G) =4$. \end{proof}

\begin{proof}[\normalfont\textbf{Proof of Lemma~\ref{L33}}]
We first show that if there exist three sets among $C_1, \dots, C_5$ such that the subgraph induced by their union is $K_3$-free, then $\chi(G) = 5$. Without loss of generality, suppose that $G[C_1\cup C_2\cup C_3]$ is $K_3$-free. Then $\omega(C_1\cup C_2\cup C_3)\leq 2$. By Lemma~\ref{LD1}, both $C_1\cup C_2\cup C_3$ and $C_0\cup C_4\cup C_5$ induce perfect graphs. Thus, we may colour $a_i$ with $i$ for $i\in [5]$, colour $C_1\cup C_2\cup C_3$ with colours from $\{4,5\}$, and colour $C_0\cup C_4\cup C_5$ with colours from $\{1,2,3\}$.

Hence, we may assume that for any three sets among $C_1, \dots, C_5$, the subgraph induced by their union contains a $K_3$. Since $k=3$, we may regard $B$ as such a $K_3$. Then, by~\eqref{E1}, each such $K_3$ has its three vertices belonging to three different sets among $C_1, \dots, C_5$. Let $X = \{x_1, x_2, x_3\} \subseteq C_1 \cup C_2 \cup C_3$ and 
$Y = \{y_1, y_4, y_5\} \subseteq C_1 \cup C_4 \cup C_5$ 
be two vertex sets such that both $X$ and $Y$ induce a $K_3$, 
where $x_i \in C_i$ for $i \in \{1, 2, 3\}$ and 
$y_j \in C_j$ for $j \in \{1, 4, 5\}$.

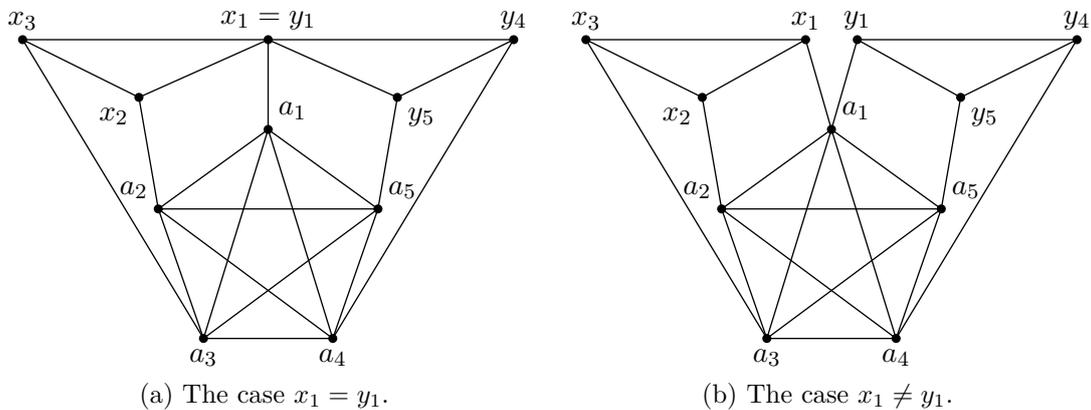
\begin{figure}[htbp]
\centering
\begin{subfigure}{0.44\linewidth}
\centering
\begin{tikzpicture}[scale=1.7, line width=0.5pt, every node/.style={font=\normalsize}]

\path[use as bounding box] (-2.1,-1.43) rectangle (2.1,1.3);

\path ( 0.00, 0.40) coordinate (a1);
\path ( 0.85, -0.22) coordinate (a5);
\path ( 0.5,-1.23) coordinate (a4);
\path (-0.5,-1.23) coordinate (a3);
\path (-0.85, -0.22) coordinate (a2);
\path (-1.9,1.10) coordinate (x3);
\path ( 0.00,1.10) coordinate (x1); 
\path ( 1.9,1.10) coordinate (y4);
\path (-1,0.65) coordinate (x2);
\path ( 1,0.65) coordinate (y5);

\draw (x3)--(x1)--(y4);
\draw (x3)--(x2)--(a2);
\draw (x3)--(a3); 
\draw (y4)--(y5)--(a5);
\draw (y4)--(a4);
\draw (x1)--(a1);
\draw (x2)--(x1);
\draw (y5)--(x1);

\foreach \A/\i in {a1/1,a2/2,a3/3,a4/4,a5/5}{
\foreach \B/\j in {a1/1,a2/2,a3/3,a4/4,a5/5}{
\ifnum\i<\j
\draw (\A)--(\B);
\fi
}
}
\foreach \p in {x2,x3,y5,y4,x1,a1,a2,a3,a4,a5}
        {
        \fill (\p) circle (1pt);
        }
        
\node[above] at (x3) {$x_3$};
\node[above] at (x1) {$x_1=y_1$};
\node[above] at (y4) {$y_4$};
\node[below right] at (y5) {$y_5$};
\node[below left] at (x2) {$x_2$};
\node[above right] at (a1) {$a_1$};
\node[above left] at (a2) {$a_2$};
\node[below] at (a3) {$a_3$};
\node[below] at (a4) {$a_4$};
\node[above right] at (a5) {$a_5$};

\end{tikzpicture}

\caption{The case $x_1=y_1$.}
\label{}
\end{subfigure}
\hspace{0.3em}
\begin{subfigure}{0.44\linewidth}
\centering
\begin{tikzpicture}[scale=1.7, line width=0.5pt, every node/.style={font=\normalsize}]

\path[use as bounding box] (-2.1,-1.43) rectangle (2.1,1.3);

\path ( 0.00, 0.40) coordinate (a1);
\path ( 0.85, -0.22) coordinate (a5);
\path ( 0.5,-1.23) coordinate (a4);
\path (-0.50,-1.23) coordinate (a3);
\path (-0.85, -0.22) coordinate (a2);
\path (-1.9,1.10) coordinate (x3);
\path ( 0.2,1.1) coordinate (y1); 
\path (-0.2,1.1) coordinate (x1);
\path ( 1.9,1.10) coordinate (y4);
\path (-1,0.65) coordinate (x2);
\path (1,0.65) coordinate (y5);

\draw (x3)--(x1);
\draw (x3)--(x2)--(a2);
\draw (x3)--(a3); 
\draw (y4)--(y5)--(a5);
\draw (y4)--(a4);
\draw (x1)--(a1);
\draw (y1)--(a1);
\draw (y1)--(y4);
\draw (y1)--(y5);
\draw (x2)--(x1);

\foreach \A/\i in {a1/1,a2/2,a3/3,a4/4,a5/5}{
\foreach \B/\j in {a1/1,a2/2,a3/3,a4/4,a5/5}{
\ifnum\i<\j
\draw (\A)--(\B);
\fi
}
}
\foreach \p in {x2,x3,y5,y4,x1,y1,a1,a2,a3,a4,a5}
        {
        \fill (\p) circle (1pt);
        }
        
\node[above] at (x3) {$x_3$};
\node[above] at (x1) {$x_1$};
\node[above] at (y1) {$y_1$};
\node[above] at (y4) {$y_4$};
\node[below right] at (y5) {$y_5$};
\node[below left] at (x2) {$x_2$};
\node[above right] at (a1) {$a_1$};
\node[above left] at (a2) {$a_2$};
\node[below] at (a3) {$a_3$};
\node[below] at (a4) {$a_4$};
\node[above right] at (a5) {$a_5$};

\end{tikzpicture}
\caption{The case $x_1\neq y_1$.}
\label{}

\end{subfigure}
\caption{Two cases in the proof of Lemma~\ref{L33}.}
\label{}
\end{figure}

If $x_1 = y_1$ (see Figure~6(a)), since $k \le 3$ and $G$ is diamond-free, $\{x_2,x_3\}$ is anticomplete to $\{y_4,y_5\}$. Hence, $\{x_2,x_3,y_4,y_5,a_5,a_1\}$ induces a $P_2\cup P_4$, a contradiction.

If $x_1 \ne y_1$ (see Figure~6(b)), since $k \le 3$ and $G$ is diamond-free, $[\{x_1, x_2\}, Y]$ is a matching in $G$.
If $\{x_1,x_2\}$ is anticomplete to $Y$, then $\{x_1,x_2,y_1,y_5,a_5,a_4\}$ induces a $P_2\cup P_4$, a contradiction. Hence, the matching $[\{x_1,x_2\}, Y]$ is nonempty. By Lemma~\ref{L0}, $G[\{x_1,x_2\}\cup Y]$ contains a $P_4$ whose vertex set contains $y_1$. Thus, $\{x_1, x_2, y_1, y_4\}$ or $\{x_1, x_2, y_1, y_5\}$ induces a $P_4$; hence, $\{a_3, a_5, x_1, x_2, y_1, y_4\}$ or $\{a_3, a_4, x_1, x_2, y_1, y_5\}$ induces a $P_2 \cup P_4$. In either case, we obtain a contradiction.

Thus, at least one of $G[C_1\cup C_2\cup C_3]$ and $G[C_1\cup C_4\cup C_5]$ is $K_3$-free, and hence  $\chi(G) =5$. \end{proof}

\section{The Case $k\geq 4$}

In this case, $\omega \geq k\geq4$. This section aims to prove the following lemma.

\begin{lemma}\label{L41}
If $G$ is a ($P_2\cup P_4$, diamond)-free graph with $\omega \geq 4$ and $k\geq 4$, then $\chi(G)=\omega$.  
\end{lemma}

To prove Lemma~\ref{L41}, we first present the following claim.

\begin{claim}\label{C3}
Let $x,y$ be two vertices in $R\cup S\cup T\cup Z$.
If $N_{A\cup B}(x)\neq N_{A\cup B}(y)$, and $|N_{A\cup B}(x)\cup N_{A\cup B}(y)|\leq 3$, then $x\nsim y$.
\end{claim}
\begin{proof}
Suppose, to the contrary, that there exist vertices $x,y\in R\cup S\cup T\cup Z$ satisfying the above conditions such that $x\sim y$. Since each of $|N_A(x)|$, $| N_A(y)|$, $|N_B(x)|$, and $|N_B(y)|$ is at most $1$, we may discuss the following two cases separately.

\medskip

\noindent\textbf{Case 1.} $|N_{A}(x)\cup N_{A}(y)|=2$ or $|N_{B}(x)\cup N_{B}(y)|=2$.

Suppose $|N_{A}(x)\cup N_{A}(y)|=2$ (the case where $|N_{B}(x)\cup N_{B}(y)|=2$ is analogous). Since $|N_{A\cup B}(x)\cup N_{A\cup B}(y)|\leq 3$, we have $|N_{B}(x)\cup N_{B}(y)|\leq 1$. By symmetry, we may assume that $x\in C_1\cap (C_0^\prime \cup C_1^\prime)$ and $y\in C_2\cap (C_0^\prime \cup C_1^\prime)$. Hence, $\{x,y\}$ is anticomplete to $\{a_3,\ldots,a_\omega,b_2,\ldots,b_k\}$. By \eqref{E1}, $[\{a_3,\ldots,a_\omega\},\{b_2,\ldots,b_k\}]$ is a matching. Since $\omega,k\geq 4$ and $G$ is $(P_2\cup P_4)$-free, Lemma~\ref{L0} implies that $\{a_3,\ldots,a_\omega\}$ is anticomplete to $\{b_2,\ldots,b_k\}$. Since $k\geq 4$, there exist vertices $b_i,b_j\in \{b_2,\ldots,b_k\}$ that are anticomplete to $a_1$. Consequently, $\{b_i,b_j,a_3,a_1,x,y\}$ induces a $P_2\cup P_4$, a contradiction.

\smallskip

\noindent\textbf{Case 2.} $|N_{A}(x)\cup N_{A}(y)|\leq 1$ and $|N_{B}(x)\cup N_{B}(y)|\leq 1$.

By symmetry, we may assume that $x,y\in (C_0\cup C_1)\cap (C_0^\prime \cup C_1^\prime)$. Thus, $\{x,y\}$ is anticomplete to $\{a_2,\ldots,a_\omega,b_2,\ldots,b_k\}$. By \eqref{E1}, $[\{a_2,\ldots,a_\omega\},\{b_2,\ldots,b_k\}]$ is a matching. Since $\omega,k\geq 4$ and $G$ is $(P_2\cup P_4)$-free, Lemma~\ref{L0} implies that $\{a_2,\ldots,a_\omega\}$ is anticomplete to $\{b_2,\ldots,b_k\}$. Since $N_{A\cup B}(x)\neq N_{A\cup B}(y)$, it follows that $N_{A}(x)\neq N_{A}(y)$ or $N_{B}(x)\neq N_{B}(y)$. Suppose $N_{A}(x)\neq N_{A}(y)$ (the case where $N_{B}(x)\neq N_{B}(y)$ is analogous). Then either $a_1\sim x$ and $a_1\nsim y$, or $a_1\nsim x$ and $a_1\sim y$. By symmetry, assume that $a_1\sim x$ and $a_1\nsim y$. Since $k\geq 4$, there exist vertices $b_i,b_j\in \{b_2,\ldots,b_k\}$ that are anticomplete to $a_1$. Consequently, $\{b_i,b_j,a_2,a_1,x,y\}$ induces a $P_2\cup P_4$, a contradiction.

\medskip

This completes the proof of Claim~\ref{C3}
\end{proof}

Then we present several useful properties of the subsets in the partition defined in Section~2. In particular, Properties (P1)–(P4) follow directly from Claim \ref{C3}.

\smallskip

\noindent(P1) $Z$ is anticomplete to $R\cup S\cup T$.

\medskip

\noindent(P2) $R$ is anticomplete to $S\cup T\cup Z$, and $R_i$ is anticomplete to $R_j$ for distinct $i,j \in [k]$.

\medskip

\noindent(P3) $T$ is anticomplete to $R\cup S\cup Z$, and $T_i$ is anticomplete to $T_j$ for distinct $i,j \in [\omega]$.

\medskip

\noindent(P4) $S$ is anticomplete to $R\cup T\cup Z$. Moreover, $S_i^p$ is anticomplete to $S_i^q$ for every $i\in[\omega]$ and distinct $p,q\in[k]$; and $S_i^p$ is anticomplete to $S_j^p$ for every $p\in[k]$ and distinct $i,j\in[\omega]$.

\medskip

\noindent(P5) Suppose that $\omega=4$ and that $S_i^p$ is not
anticomplete to $S_j^q$. Then $i\neq j$ and $p\neq q$. Let
$I=\{i,j\}$ and $P=\{p,q\}$, and, for $L\subseteq[4]$, let
$A_L=\{a_r:r\in L\}$ and $B_L=\{b_r:r\in L\}$. We also write
$\overline I=[4]\setminus I$ and $\overline P=[4]\setminus P$.
Then one of the following holds:

\begin{enumerate}
\item $[A_{\overline I},B_{\overline P}]$ is a perfect matching,
and
$[A_I,B_{\overline P}]=[A_{\overline I},B_P]=\emptyset$;

\item both $[A_I,B_{\overline P}]$ and
$[A_{\overline I},B_P]$ are perfect matchings, and
$[A_I,B_P]=[A_{\overline I},B_{\overline P}]=\emptyset$.
\end{enumerate}

\begin{proof}
Choose $x\in S_i^p$ and $y\in S_j^q$ such that $x\sim y$. By
(P4), $i\neq j$ and $p\neq q$. Moreover, $\{x,y\}$ is
anticomplete to $A_{\overline I}\cup B_{\overline P}$. Since
$[A,B]$ is a matching, let
$d=|[A_{\overline I},B_{\overline P}]|\in\{0,1,2\}$. If $d=1$,
then $G[A_{\overline I}\cup B_{\overline P}]$ is an induced
$P_4$, which, together with the edge $xy$, induces a
$P_2\cup P_4$, a contradiction. If $d=2$, then
$[A_{\overline I},B_{\overline P}]$ is a perfect matching.
Since $[A,B]$ is a matching, it follows that
$[A_I,B_{\overline P}]=[A_{\overline I},B_P]=\emptyset$, and
hence (a) holds.

It remains to consider the case $d=0$. We first show that
$[A_I,B_{\overline P}]$ is a perfect matching. Suppose, to the
contrary, that $|[A_I,B_{\overline P}]|\leq1$. Then there exists
$a_h\in A_I$ that is anticomplete to $B_{\overline P}$. Let $z$
be the vertex in $\{x,y\}$ adjacent to $a_h$, let $z'$ be the
other vertex, and choose any $a_\ell\in A_{\overline I}$. Then
$a_\ell-a_h-z-z'$ is an induced $P_4$, while the two vertices
in $B_{\overline P}$ induce an edge anticomplete to this $P_4$.
This gives an induced $P_2\cup P_4$, a contradiction. Therefore,
$[A_I,B_{\overline P}]$ is a perfect matching. By exchanging the
roles of $A$ and $B$, $[A_{\overline I},B_P]$ is also a perfect
matching. These two perfect matchings saturate all vertices of
$A\cup B$, and hence
$[A_I,B_P]=[A_{\overline I},B_{\overline P}]=\emptyset$. Thus,
(b) holds.
\end{proof}

\noindent(P6) Suppose that $\omega=4$ and that $S_i^p$ is stable
for all $i,p\in[4]$. Then the vertex set of each component of
$G[S]$ can be partitioned into at most four stable sets, each
contained in a distinct $S_i^p$, such that all lower indices and
all upper indices are pairwise distinct.

\begin{proof}
Define an auxiliary graph $\Gamma$ whose vertices are the nonempty
sets $S_i^p$, where two such sets are adjacent in $\Gamma$ if and
only if they are not anticomplete in $G$. By relabelling the
vertices of $A$ and $B$, we may assume that
$[A,B]=\{a_tb_t:t\in[m]\}$ for some $m\in[4]$. Let $S_i^p$ and
$S_j^q$ be adjacent in $\Gamma$, and let $I=\{i,j\}$ and
$P=\{p,q\}$. By (P5)(a), we have
$\overline I=\overline P\subseteq[m]$, while by (P5)(b), we have
$m=4$ and $P=\overline I$. Thus, every edge of $\Gamma$ satisfies
one of these two conditions.

If $m=1$, then $\Gamma$ has no edges. If $m=2$, then the only
possible edges are $S_3^3S_4^4$ and $S_3^4S_4^3$. If $m=3$, then
the only possible edges are $S_r^rS_4^4$ and $S_r^4S_4^r$ for
$r\in[3]$. Hence, every component of $\Gamma$ containing an edge
is contained in $\{S_3^3,S_4^4\}$ or $\{S_3^4,S_4^3\}$ when
$m=2$, and is contained in
$\{S_1^1,S_2^2,S_3^3,S_4^4\}$ or in one of
$\{S_r^4,S_4^r\}$, where $r\in[3]$, when $m=3$. Finally, suppose
that $m=4$. By the two conditions above, every edge of $\Gamma$
has both ends in one of the following four sets:
\[
\begin{aligned}
\mathcal D_0&=\{S_1^1,S_2^2,S_3^3,S_4^4\},&
\mathcal D_1&=\{S_1^2,S_2^1,S_3^4,S_4^3\},\\
\mathcal D_2&=\{S_1^3,S_2^4,S_3^1,S_4^2\},&
\mathcal D_3&=\{S_1^4,S_2^3,S_3^2,S_4^1\}.
\end{aligned}
\]
Therefore, every component of $\Gamma$ containing an edge is
contained in one of $\mathcal D_0,\mathcal D_1,\mathcal D_2$,
and $\mathcal D_3$. In every case, each component of $\Gamma$
contains at most four sets $S_i^p$, and their lower indices and
upper indices are pairwise distinct.

Let $C$ be a component of $G[S]$. Since every $S_i^p$ is stable,
each edge of a path in $C$ joins vertices belonging to two
different sets $S_i^p$, and these two sets are adjacent in
$\Gamma$. It follows that all sets $S_i^p$ intersecting $V(C)$
belong to the same component of $\Gamma$. Consequently, the
nonempty sets $V(C)\cap S_i^p$ form a partition of $V(C)$ into
at most four stable sets, with all lower indices and all upper
indices pairwise distinct.
\end{proof}

\noindent(P7) If $\omega\geq5$, then $S_i^p$ is anticomplete to $S_j^q$ for distinct $p,q \in [k]$ and distinct $i,j \in [\omega]$.
\begin{proof}
By symmetry, it suffices to show that $S_1^1$ is anticomplete to $S_2^2$. Suppose to the contrary that there exist vertices $x\in S_1^1$ and $y\in S_2^2$ such that $x\sim y$. By \eqref{E1}, $[\{a_3,\ldots,a_\omega\},\{b_3,\ldots,b_k\}]$ is a matching. Since $\omega\ge 5$, $k\ge 4$ and $G$ is $(P_2\cup P_4)$-free, Lemma~\ref{L0} implies that $\{a_3,\ldots,a_\omega\}$ is anticomplete to $\{b_3,\ldots,b_k\}$. Since $\omega\geq 5$, there exist vertices $a_i,a_j\in \{a_3,\ldots,a_\omega\}$ that are anticomplete to $b_1$. Consequently, $\{a_i,a_j,b_3,b_1,x,y\}$ induces a $P_2\cup P_4$, a contradiction.
\end{proof}

\noindent(P8) If there exists an $S_i^p$ that is not stable, then $[A,B]=\{a_ib_p\}$, and every set $S_j^q$ with $(j,q)\neq (i,p)$ is stable.
\begin{proof}
By symmetry, it suffices to show that (P8) holds in the case that $S_1^1$ is not stable. Assume that there exist two adjacent vertices $x,y\in S_1^1$. Then $a_1\sim b_1$, since $G$ is diamond-free.  By ~\eqref{E1}, $[\{a_2,\ldots,a_\omega\},\{b_2,\ldots,b_k\}]$ is a matching. Since $\omega,k\geq 4$ and $G$ is $(P_2\cup P_4)$-free, Lemma~\ref{L0} implies that $\{a_2,\ldots,a_\omega\}$ is anticomplete to $\{b_2,\ldots,b_k\}$. Hence, $[A,B]=\{a_1b_1\}$. Since $G$ is diamond-free, every set $S_j^q$ with $(j,q)\neq (1,1)$ is stable.
\end{proof}

\noindent(P9) If there exists an $S_i^p$ that is not stable, then $S_j^q$ is anticomplete to $S_{j^\prime}^{q^\prime}$ for distinct $q,q^\prime \in [k]$ and distinct $j,j^\prime \in [\omega]$.
\begin{proof}
If $\omega\ge 5$, then (P9) follows directly from (P7). 
We therefore consider the case $\omega=4$ and $k=4$. 
By symmetry, assume that $i=p=1$. By (P8), we have $[A,B]=\{a_1b_1\}$. 
Suppose, to the contrary, that there exist vertices 
$x\in S_j^q$ and $y\in S_{j'}^{q'}$ such that $x\sim y$, 
where $j\neq j'$ and $q\neq q'$.
Since $[A,B]=\{a_1b_1\}$ and $\omega=k=4$, we have
\[
|[A\setminus\{a_j,a_{j'}\},\, B\setminus\{b_q,b_{q'}\}]|\le 1.
\]
Thus, there exists a vertex 
$w\in A\setminus\{a_j,a_{j'}\}$, with $w\neq a_1$, 
that is anticomplete to $B\setminus\{b_q,b_{q'}\}$. 
If $a_j\neq a_1$, then $\{w,a_j,x,y\}\cup(B\setminus\{b_q,b_{q'}\})$ induces a $P_2\cup P_4$. 
If $a_{j'}\neq a_1$, then $\{w,a_{j'},x,y\}\cup(B\setminus\{b_q,b_{q'}\})$ induces a $P_2\cup P_4$. 
In either case, we obtain a contradiction.
\end{proof}

With these properties in hand, we now prove Lemma~\ref{L41}.

\begin{proof}[\normalfont\textbf{Proof of Lemma~\ref{L41}}]

Let $G$ be a $(P_2\cup P_4,\text{ diamond})$-free graph with $\omega \ge 4$ and $k \ge 4$. By~\eqref{E2}, $A$ is not anticomplete to $B$.  
Then, by~\eqref{E1}, we may relabel $A$ and $B$ so that, for some integer $m\in [k]$, we have $a_i\sim b_i$ for all $i\in [m]$, and $a_p\nsim b_q$ whenever $p\in [\omega]\setminus [m]$ and $q\in [k]\setminus [m]$. 
By Lemma~\ref{LD1}, each of the sets $R_j$, $T_i$, and $Z$ induces a perfect graph. Since $\omega(R_j), \omega(T_i) \leq \omega - 1$ and $\omega(Z) \le k\leq \omega$, we may colour them with at most $\omega - 1$, $\omega - 1$, and $\omega$ colours, respectively. We now construct an $\omega$-colouring of $G$ by distinguishing two cases according to whether every $S_i^p$  is stable. 

\medskip
\noindent\textbf{Case 1.} $S_i^p$  is stable for all $i\in [\omega]$ and $p\in [k]$.
\medskip

We first colour $A$ and $B$.
\begin{itemize}
\item Colour $a_i$ with colour $i$ for $i\in [\omega]$.
\item if $m=1$, colour $b_i$ with colour $i+1$ for $i\in [k-1]$, and colour $b_k$ with colour $1$;
\item if $m\ge2$, colour $b_i$ with colour $i+1$ for $i\in [m-1]$, colour $b_m$ with colour $1$, and colour $b_j$ with colour $j$ for $j\in [k]\setminus [m]$.
\end{itemize}

Next, we colour $S$, considering separately the cases $\omega = 4$ and $\omega \ge 5$.
\begin{itemize}
\item If $\omega =4$, by (P6), we may colour each component of $G[S]$ with at most four suitable colours from $[4]$. To justify this, we construct a bipartite graph demonstrating the existence of such a colouring. 
Let $X$ denote the set of stable sets in a component of $G[S]$ as described in (P6). 
For each $X_t \in X$, let $Y_t \subseteq [4]$ be the set of colours that can be assigned to $X_t$; 
for example, if $X_t \subseteq S_i^j$, then $Y_t = [4] \setminus (c(a_i) \cup c(b_j))$. 
Let 
\[
Y = \bigcup_{X_t \in X} Y_t.
\] 
We now regard $X$ and $Y$ as the two vertex sets of a bipartite graph $Q$ with bipartition $(X, Y)$. 
Each vertex $X_t\in X$ represents a stable set in the component, and each vertex $y\in Y$ represents a colour in $[4]$. 
Vertices $X_t$ and $y$ are adjacent in $Q$ if and only if $y\in Y_t$; equivalently, $N_Q(X_t)=Y_t$ for all $X_t\in X$.

It remains to verify Hall's condition. Let $\mathcal{X}\subseteq X$.
By (P6), the lower indices corresponding to the members of $\mathcal{X}$
are pairwise distinct, and so are their upper indices. Moreover, the
colourings of both $A$ and $B$ are injective. Consequently, each colour
in $[4]$ is excluded from at most two of the sets $Y_t$: at most once
as a colour of a vertex of $A$, and at most once as a colour of a vertex
of $B$. Also, $|Y_t|\geq 2$ for every $X_t\in X$.

If $|\mathcal{X}|\leq 2$, then
\[
\left|\bigcup_{X_t\in\mathcal{X}}Y_t\right|\geq |\mathcal{X}|.
\]
If $|\mathcal{X}|=3$ and the above union has size at most two, then
some colour is excluded from all three sets $Y_t$, a contradiction.
If $|\mathcal{X}|=4$ and the above union has size at most three, then
some colour is excluded from all four sets $Y_t$, again a contradiction.
Thus Hall's condition holds. By Hall’s Theorem~\cite{HALL} applied to $Q$, there exists a matching saturating $X$.

\item If $\omega \geq 5$, by (P4) and (P7), we may colour each $S_i^j$ with at most one colour from $[\omega]\setminus (c(a_i)\cup c(b_j))$ for $i\in [\omega]$ and $j\in [k]$.
\end{itemize}
Finally, we colour $Z$, $R$, and $T$ as follows:
\begin{itemize}
\item By (P1), colour $Z$ with at most $\omega$ colours from $[\omega]$.
\item By (P2), colour $R_i$ with at most $\omega -1$ colours from $[\omega]\setminus c(b_i)$ for $i\in [k]$.
\item By (P3), colour $T_i$ with at most $\omega -1$ colours from $[\omega]\setminus c(a_i)$ for $i\in [\omega]$.
\end{itemize}

\noindent\textbf{Case 2.} There exists an $S_i^p$ that is not stable.
\medskip

By (P8) together with the argument given in the first paragraph of the proof,
it follows that $i=p=m=1$, $[A,B]=\{a_1b_1\}$, and every set $S_j^q$ with $(j,q)\neq(1,1)$ is stable. Note that $\omega(S_1^1)\leq \omega-2$ and $S_1^1$ is perfect, we may colour it with at most $\omega -2$ colours.

\begin{itemize}
    \item Colour $a_i$ with colour $i$ for $i\in [\omega]$.
    \item Colour $b_i$ with colour $i+1$ for $i\in [k-1]$, and colour $b_k$ with colour $1$.
    \item By (P4) and (P9), colour $S_1^1$ with at most $\omega-2$ colours from $[\omega] \setminus \{1,2\}$, and colour $S_j^q$ with at most one colour from $[\omega]\setminus (c(a_j)\cup c(b_q))$ for $(j,q)\neq (1,1)$.
    \item By (P1), colour $Z$ with at most $\omega$ colours from $[\omega]$.
    \item By (P2), colour $R_i$ with at most $\omega -1$ colours from $[\omega]\setminus c(b_i)$ for $i\in [k]$.
    \item By (P3), colour $T_i$ with at most $\omega -1$ colours from $[\omega]\setminus c(a_i)$ for $i\in [\omega]$.
\end{itemize}
This completes the proof of Lemma~\ref{L41}.
\end{proof}

\section{Proof of Theorem~\ref{THM1}}
\begin{proof}[\normalfont\textbf{Proof of Theorem~\ref{THM1}}]

Let $G$ be a ($P_2\cup P_4$, diamond)-free graph.  
If $\omega \le 2$, then $\chi(G)\le 4$ by Lemma~\ref{LP2P4}; 
thus we may assume $\omega \ge 3$.  
If $k \le 2$, then the theorem holds by Lemmas~\ref{LD2} and~\ref{LD21}; 
hence we may further assume $k \ge 3$.

If $\omega=3$ and $k=3$, then Lemma~\ref{L31} implies $\chi(G)\le 6$.  

If $\omega\ge 4$ and $k=3$, then Lemmas~\ref{LD2}, \ref{L32}, and~\ref{L33} imply $\chi(G)=\omega$.  

If $\omega\ge 4$ and $k\ge 4$, then Lemma~\ref{L41} implies $\chi(G)=\omega$.

Next, we show that the bounds are optimal for all $\omega \ge 2$ by presenting two examples for $\omega=2$ and $\omega=3$ that attain the corresponding upper bounds.

Let $H_1$ be the Mycielski--Gr\"otzsch graph (see Figure 7(a)) and let $H_2$ be the complement of the Schl\"afli graph (see Figure 7(b); also available at https://houseofgraphs.org/graphs/19273). 
Both $H_1$ and $H_2$ are $(P_2 \cup P_4,\text{ diamond})$-free. 
Moreover, $\omega(H_1)=2$ and $\chi(H_1)=4$, while $\omega(H_2)=3$ and $\chi(H_2)=6$.

This completes the proof of Theorem~\ref{THM1}. \end{proof}

\begin{figure}[htbp]
    \centering
    \begin{subfigure}{0.39\linewidth}
        \centering
        \includegraphics[width=\linewidth]{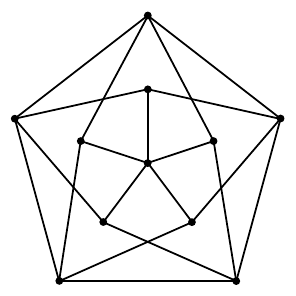}
        \caption{The Mycielski--Gr\"otzsch graph.}
        \label{}
    \end{subfigure}
\hspace{4em}
    \begin{subfigure}{0.39\linewidth}
        \centering
        \includegraphics[width=\linewidth]{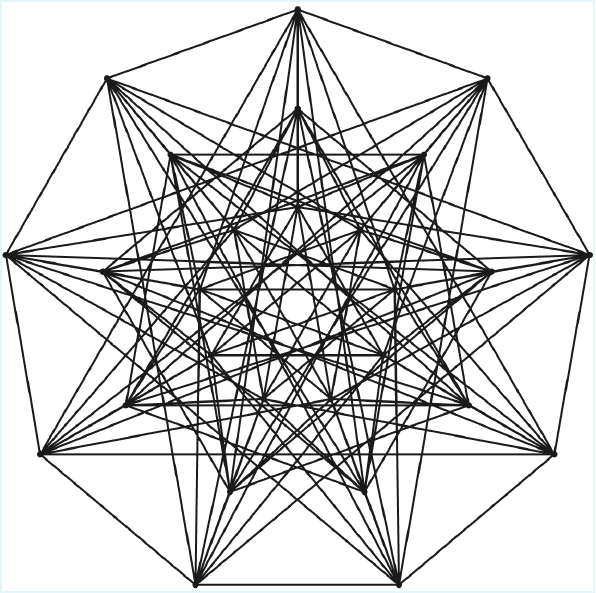}
        \caption{The complement of the Schl\"afli graph.}
        \label{fig:right}
    \end{subfigure}
    \caption{The graphs $H_1$ and $H_2$ used to show the optimality of the bounds in Theorem~\ref{THM1}.}
    \label{}
\end{figure}

\section*{Declaration of competing interest}

The authors declare that they have no known competing financial interests or personal relationships that could have appeared to influence the work reported in this paper.

\section*{Acknowledgements}

The research was partially supported by RGC Competitive Earmarked Research Grant 16308821.

\end{document}